\documentclass[11pt]{amsart}

\usepackage{graphicx}
\usepackage{amsmath,amssymb,amsfonts}
\usepackage{amsthm}
\usepackage{mathrsfs}
\usepackage{tikz-cd}
\usepackage{hyperref}

\newcommand{\OG}{\mathcal{O}_G}
\newcommand{\OH}{\mathcal{O}_H}
\newcommand{\Cic}{\underline{\mathcal{C}}}
\newcommand{\Cs}{\underline{\mathcal{C}}^{\otimes}}
\newcommand{\Ci}{\mathcal{C}}
\newcommand{\Ztwo}{C_2}

\newcommand{\op}{\operatorname{op}}
\newcommand{\BOG}{\operatorname{BO}_n(G)}
\newcommand{\Top}{\operatorname{Top}}
\newcommand{\TopG}{\underline{\Top}^G}
\newcommand{\Mfld}{\operatorname{Mfld}}
\newcommand{\fr}{\operatorname{-fr}}
\newcommand{\MG}{\underline{\Mfld}^G}
\newcommand{\MGB}{\underline{\Mfld}^{G,f \fr}}
\newcommand{\MGs}{\underline{\Mfld}^{G,\sqcup}}
\newcommand{\MGBs}{\underline{\Mfld}^{G,f\fr,\sqcup}}
\newcommand{\Fin}{\operatorname{Fin}}
\newcommand{\FinG}{\underline{\Fin}^G_{\ast}}
\newcommand{\FinH}{\underline{\Fin}^H_{\ast}}
\newcommand{\Disk}{\operatorname{Disk}}
\newcommand{\DG}{\underline{\Disk}^G}
\newcommand{\DGB}{\underline{\Disk}^{G,f\fr}}
\newcommand{\DGs}{\underline{\Disk}^{G,\sqcup}}
\newcommand{\DGBs}{\underline{\Disk}^{G,f\fr,\sqcup}}
\newcommand{\Rep}{\operatorname{Rep}}
\newcommand{\RG}{\underline{\Rep}_n (G)}
\newcommand{\RGs}{\underline{\Rep}^{\sqcup}_n (G)}
\newcommand{\RGBs}{\underline{\Rep}^{f\fr,\sqcup}_n (G)}
\newcommand{\crt}{\operatorname{-}}
\newcommand{\OGTop}{\mathcal{O}_G\crt\Top}
\newcommand{\RHVs}{\underline{\Rep}^{V\fr,\sqcup}_n (H)}
\newcommand{\RGHs}{\underline{\Rep}^{G/H\fr,\sqcup}_n (G)}
\newcommand{\RHV}{\underline{\Rep}^{V\fr}_n (H)}
\newcommand{\RGH}{\underline{\Rep}^{G/H\fr}_n (G)}
\newcommand{\DGH}{\mathcal{D}^{G,G/H\fr}}
\newcommand{\DHV}{\mathcal{D}^{H,V\fr}}
\newcommand{\Fun}{\operatorname{Fun}}
\newcommand{\FunG}{\Fun_G}
\newcommand{\FunGs}{\FunG^{\otimes}}
\newcommand{\Sp}{\operatorname{Sp}}
\newcommand{\Spp}{\underline{\Sp}}
\newcommand{\Otwo}{O(2)}
\newcommand{\THR}{\operatorname{THR}}
\newcommand{\THH}{\operatorname{THH}}
\newcommand{\Map}{\operatorname{Map}}
\newcommand{\Res}{\operatorname{Res}}
\newcommand{\Orbit}{\operatorname{Orbit}}
\newcommand{\Alg}{\operatorname{Alg}}
\newcommand{\Env}{\operatorname{Env}}

\theoremstyle{thmstyleone}
\newtheorem{theorem}{Theorem}[section]
\newtheorem{proposition}[theorem]{Proposition}
\newtheorem{corollary}[theorem]{Corollary}
\newtheorem{lemma}[theorem]{Lemma}

\newtheorem{step}{Step}

\theoremstyle{thmstyletwo}
\newtheorem{example}[theorem]{Example}
\newtheorem{remark}[theorem]{Remark}

\theoremstyle{thmstylethree}
\newtheorem{definition}[theorem]{Definition}

\newtheorem{construction}[theorem]{Construction}
\newtheorem{mydescription}[theorem]{Description}

\title{Universal property of framed $G$-disc algebras}
\author{Aleksandar Miladinović}

\begin{document}

\begin{abstract}
Given a compact Lie group $G$ and its finite subgroup $H$ we prove that the $\infty$-category of $G/H$-framed $G$-disc algebras taking values in a $G$-symmetric monoidal category $\Cs$ is equivalent to the $\infty$-category of $V$-framed $H$-disc algebras (where $V$ is an $H$-representation) which take values in $\Cs_H$, the underlying $H$-symmetric monoidal subcategory of $\Cs$. We will use this construction to refine the $C_2$-action on the real topological Hochschild homology to an $O(2)$-action.
\end{abstract}

\maketitle
\tableofcontents

\section{Introduction}
\addtocontents{toc}{\protect\setcounter{tocdepth}{1}}

As equivariant version of discs, $G$-discs represent fundamental objects of the equivariant version of factorization homology. Factorization homology, developed by Ayala and Francis (\cite{AF15}), or chiral homology, defined by Lurie (\cite{HA}), is an invariant of a geometric and algebraic input. The equivariant version of this construction is given by Horev \cite{AH19} for finite groups and subsequently by the author in \cite{Mil22} for $G$ a compact Lie group (where the orbit category consists of those $G$-orbits with finite stabilizers). The geometric input of factorization homology is given by $n$-dimensional (framed) $G$-manifolds (where $n$ is a fixed number), while the algebraic input is given by (framed) $G$-disc algebras. In such construction $G$-discs represent the meeting point of algebra and geometry of a $G$-manifold. On one hand, $G$-discs give us an insight into the geometry of $G$-manifolds by capturing local properties. This can be seen through their link with the equivariant configuration spaces. On the other hand, $G$-discs encode $G$-disc algebras, which are $G$-symmetric monoidal functors with the source being the $G$-$\infty$-category of (framed) $G$-discs taking values in some $G$-symmetric monoidal category $\Cs$. These algebras are used as coefficients in the equivariant factorization homology.

This paper is dedicated to proving a universal property of the $G$-disc algebras, where $G$ is a compact Lie group. Informally, we show that the $\infty$-category of $G$-disc algebras with coefficients in some $G$-symmetric monoidal category $\Cs$ is equivalent to the $\infty$-category of $H$-disc algebras taking values in the underlying $H$-symmetric monoidal subcategory $\Cs_H$ of $\Cs$ with a suitable choice of framing on both $G$-discs and $H$-discs. In other words, the $G$-symmetric monoidal category of $G$-discs is freely generated by the $H$-symmetric monoidal category of $H$-discs with the suitable choice of framing on both $G$-discs and $H$-discs.

\subsection*{Framing on $G$-discs and $H$-discs}

Let $M$ be an $n$-dimensional $G$-manifold and let $f: B\rightarrow\BOG$ be a $G$-map. By $f$-framing (or $B$-framing) on $M$ we mean a $G$-map $f_M: M\rightarrow B$ which fits into a $G$-homotopy commutative diagram
\begin{center}
\begin{tikzcd} [row sep=2em, column sep=2em]
 & B \arrow[dr,"f"] & \\
M \arrow[ur,"f_M"]\arrow[rr,"\tau_M"] & & \BOG
\end{tikzcd}
\end{center}
where $\tau_M$ is the $G$-tangent bundle classifying map of $M$. If $B=\ast$ (i.e. $B$ is a point), the framing corresponds to the trivialization of the tangent bundle $TM\cong M\times V$ where $V$ is a $G$-representation. In such case we write $V$-framing instead of $f$-framing.

As we have stated in the beginning, there is a suitable choice of framing on $G$-discs and $H$-discs under which we obtain the equivalence of $\infty$-categories of $G$-disc algebras and $H$-disc algebras. Namely, the $H$-discs are framed over a point, which corresponds to the $V$-framing where $V$ is an $H$-representation. Additionally, this framing map $\ast\rightarrow\operatorname{BO}_n(H)$ is adjoint to the framing map $G/H\rightarrow\BOG$ of $G$-discs.

\subsection*{Main result}

The most important result of this paper is Theorem \ref{universalthm} which we state here as Theorem \ref{universalthmintro}.
\begin{theorem}
\label{universalthmintro}
Let $\Cs$ be a $G$-symmetric monoidal category. Then the $G$-sym\-metric monoidal category of $G/H$-framed $G$-discs is freely generated by the $H$-symmetric monoidal category of $V$-framed $H$-discs. In other words, there is an equivalence
\[
\FunGs(\underline{\Disk}^{G,G/H\fr},\Cs) \xrightarrow{\simeq} \Fun^{\otimes}_H (\underline{\Disk}^{H,V\fr},\Cs_H)
\]
where $\Cs_H$ is the underlying $H$-symmetric monoidal subcategory of $G$-sym\-metric monoidal category $\Cs$, and where $V$ is an $H$-representation.
\end{theorem}

\textbf{What would be the motivation for showing this universal property?} The input of the $G$-factorization homology is a $G$-disc algebra together with a $G$-manifold which we regard as a genuine $G$-object in the parametrized $\infty$-category of $G$-manifolds, and as an output we receive a genuine $G$-object. The universal property allows us to replace $G$-algebra with an $H$-algebra, given the suitable framings of both $G$-discs and $H$-discs. To be more precise, a $V$-framed $H$-disc algebra $A_H$ in $\Cs_{H}$ (the underlying $H$-$\infty$-category of a $G$-$\infty$-category $\Cs$), gives us the corresponding $G/H$-framed $G$-algebra $A$ in $\Cs$. Moreover, computing $\int_{G/H} A$ now gives us a genuine $G$-object (i.e. a coCartesian section of $\Cs\rightarrow \OG^{\op}$) living in $\Cs$. To sum up, we have the following maps
\begin{center}
\begin{tikzcd} [row sep=2em, column sep=3em]
\{ \text{$V$-framed $H$-disc algebras in $\Cs_H$} \} \arrow[d] \\
\{ \text{$G/H$-framed $G$-disc algebras in $\Cs$} \} \arrow[d]\\
\{ \text{Genuine $G$-objects in $\Cs$} \}
\end{tikzcd}
\end{center}

Additionally, let $G=\Otwo$ and $H=\Ztwo$ with $\Cs = \Spp^{\Otwo}$ being the $G$-symmetric monoidal category of genuine $\Otwo$-spectra. Furthermore, let $A_H$ be a $\sigma$-framed $C_2$-algebra with coefficients in $\Spp^{C_2}$ where $\sigma$ is a one-dimensional sign representation. Then, by \cite[7.1.2]{AH19} the factorization homology of $S^1$ with a group $C_2$ acting by reflection gives the following equivalence $\int_{S^1} A_H \simeq \THR(A_H)$, where $\THR(A_H)$ is the topological real Hochschild homology of $A_H$. Using the universal property, we may obtain an $O(2)$-algebra $A$ corresponding to $A_H$. Moreover, considering $S^1$ with an action of $O(2)$, the factorization homology $\int_{S^1} A$ produces a genuine $O(2)$-object in $\underline{\Sp}^{O(2)}$ which represents real topological Hochchild homology of $A_H$ with a refined $O(2)$-action.

To conclude, the result on the universal property of framed $G$-discs may carve a path to a new insight into the norm maps of Hill, Hopkins and Ravenel \cite{HHR16}.

\subsection*{Organization by sections:}

\begin{enumerate}
\setcounter{enumi}{1}

\item \textbf{Preliminaries}: This section is expository in nature and contains a recollection of results on parametrized higher category theory, $G$-$\infty$-category of $G$-manifolds and consequently the definition and the description of the $G$-$\infty$-category of $G$-discs and its framed variants.

\item \textbf{$G$-approximations to $G$-$\infty$-operads}: In this section we develop the theory of $G$-approximations to $G$-$\infty$-operads, and consequently, prove Proposition \ref{ha23323}, the main result of this section.

\item \textbf{The universal property}: We use the theory of $G$-approximations to prove the main result, Theorem \ref{universalthm}.

\item \textbf{Applications}: In the final section, we give two examples as applications of the universal property.
\begin{itemize}
\item Associative algebra objects with genuine involution in $\underline{\Sp}^{\Ztwo}$ correspond to the $O(2)$-genuine objects in $\underline{\Sp}^{O(2)}$. In addition, we will see how to refine the $\Ztwo$-genuine structure on the real topological Hochschild homology to $O(2)$-genuine structure;
\item Associative algebra objects in the $\infty$-category of spectra $\Sp$ correspond to the $S^1$-genuine objects in $\underline{\Sp}^{S^1}$.
\end{itemize}
\end{enumerate}

\subsection*{Acknowledgements}

This paper is a revised version of the third part of the author's thesis. The author would like to thank Yonatan Harpaz for the introduction to this topic as well as for his insights and patience, Denis Nardin for his advices and recommendations. Additionally, the author would like to extend the graditude to all the people of the algebraic topology team at University Paris $13$. Finally, the author would like to than Branislav Prvulović from the University of Belgrade on his useful comments and conversations.

\subsection*{Funding}
This research received no external funding.

\addtocontents{toc}{\protect\setcounter{tocdepth}{2}}

\section{Preliminaries}

This section is expository in nature. It contains basic information on the higher parametrized category theory (see \cite{BDGNS16,DN16,DN17,JS18,NS22}) as well as on the construction and the definition of the $G$-$\infty$-category of $G$-manifolds and $G$-discs and their framed variants (see \cite{AH19,Mil22}).

\subsection{$G$-$\infty$-categories}

\begin{definition}
Let $T$ be a small $\infty$-category. A $T$-parametrized or $T$-$\infty$-category is a coCartesian fibration
\begin{equation*}
p: C\rightarrow T^{\op}
\end{equation*}
A $T$-parametrized functor, or $T$-functor, between two $T$-$\infty$-categories $p:C\rightarrow T^{\op}$ and $q:D\rightarrow T^{\op}$ is a functor $f:C\rightarrow D$ over $T^{\op}$ which sends $p$-coCartesian arrows to $q$-coCartesian arrows.

In the case when $p:C\rightarrow T^{\op}$ is a left fibration we say that $C$ is a $T$-space or $T$-$\infty$-groupoid. A $T$-functor between two $T$-spaces is called a $T$-map.
\end{definition}

\begin{example}
\label{prex32}
Let $G$ be a finite group and let $\mathcal{O}_G$ be a full subcategory of the category of $G$-spaces spanned by the transitive $G$-spaces ($G$-orbits). By abuse of notation, we will mark with $\OG$ the nerv $N(\OG)$. We call $\mathcal{O}_G$ the orbit category, and we call $\OG$-$\infty$-category $p:C\rightarrow \OG^{\op}$ simply a $G$-$\infty$-category. We can look at the objects $O\in\OG$ as cosets $G/H$ by choosing a basepoint $x\in O$ with $\operatorname{Stab}_x (O) = H$. Informally, by Lurie's straightening/unstraightening, a $G$-$\infty$-category $p:C\rightarrow \OG^{\op}$ is classified by a functor $\mathcal{C}:\OG^{\op} \rightarrow Cat_{\infty}$ which sends every orbit $G/H$ to the fiber $C_{[G/H]} := p^{-1}(G/H)$.
\end{example}
\begin{definition}
(\cite[1.2 and 2.2]{DN17}) Let $T$ be a small $\infty$-category. We define the $\infty$-category of finite $T$-sets denoted as $F_T$ to be the subcategory of $Fun(T^{\op},\Top)$ spanned by finite coproducts of representables. In the case when $T = \OG$, $F_{\OG} : = F_G$ corresponds to the category of finite $G$-sets.

We say that $T$ is \emph{orbital} if $F_T$ admits all pullbacks and we say that $T$ is \emph{cartesian orbital} if $F_T$ admits all limits.

To add on, we say that an orbital $\infty$-category $T$ is \emph{atomic} if there are no nontrivial retracts in $T$ i.e. every map with a left inverse is an equivalence.
\end{definition}
\begin{example}
Let $G$ be a finite group, then $\OG$ is a cartesian orbital $\infty$-category which is also atomic.

If we take $G$ to be a compact Lie group then $\OG$ is not orbital in general, but the $\infty$-category of transitive $G$-spaces with finite stabilizers is atomic orbital $\infty$-category (see \cite[ex. 1.3]{DN17}). It is cartesian if $G$ is a finite group.

These properties of the orbit category $\OG$ are of great importance since they allow us to define the $G$-symmetric monoidal structure.
\end{example}

We will now go through several examples of $G$-$\infty$-categories that play a role in this paper. Throughout the paper $G$ is taken to be a compact Lie group and $\OG$ is taken to be a full subcategory of $G$-spaces spanned by those transitive $G$-spaces with finite stabilizers:

\addtocontents{toc}{\protect\setcounter{tocdepth}{1}}

\subsection*{Example: $G$-$\infty$-category of $G$-spaces}

In order to define the $G$-$\infty$-category of $G$-spaces, we will first define a category $\OGTop$ whose objects are $G$-maps $X\rightarrow O$ where $X$ is a $G$-$CW$-complex and $O\in\OG$ and whose morphisms are $G$-commutative diagrams
\begin{center}
\begin{tikzcd} [row sep=2em, column sep=2em]
X_1 \arrow[d]\arrow[r] & X_2 \arrow[d] \\
O_1 \arrow[r] & O_2
\end{tikzcd}
\end{center}
where $(X_1\rightarrow O_1),(X_2\rightarrow O_2)\in\OGTop$. We can regard $\OGTop$ as a topological category by taking the space of morphisms to be a subspace of
\[
\Map_{\OGTop} (X_1\rightarrow O_1,X_2\rightarrow O_2) \subseteq \Map_{\Top^G} (X_1,X_2) \times \Map_{\OG}(O_1,O_2)
\]
consisting of those maps such that upper diagram is commutative.

The idea behind this construction is that we can regard $X\rightarrow O$ as an object classifying an $H$-space $X_H$ where we write $O\cong G/H$ with the choice of a basepoint of $O$, and where $X_H$ is the fiber of $X\rightarrow G/H$ over $eH$ (with $e$ being the neutral element of $G$).

Consider the forgetful functor $q:\OGTop\rightarrow\OG$ which, after applying the topological nerve functor becomes
\[
N(q): N(\OGTop)\rightarrow\OG
\]
where, by abuse of notation we write $\OG$ for $N(\OG)$. This functor is a Cartesian fibration where the Cartesian arrows are given by diagrams as above which are pullback squares. Then by dualizing this Cartesian fibration (see \cite{BGN14}) we obtain our $G$-$\infty$-category of $G$-spaces which we will mark with $\TopG$:
\[
((N(\OGTop))^{\vee}\rightarrow \OG^{\op})\cong (\TopG \rightarrow \OG^{\op})
\]
The elements of $\TopG$ can still be written as $G$-maps $X\rightarrow O$ whereas the morphisms between $X_1\rightarrow O_1$ and $X_2\rightarrow O_2$ are represented by diagrams
\begin{center}
\begin{tikzcd} [row sep=2em, column sep=2em]
X_1 \arrow[d] & X\arrow[d]\arrow[l]\arrow[r] & X_2 \arrow[d] \\
O_1 & O_2 \arrow[l]\arrow[r,"="] & O_2
\end{tikzcd}
\end{center}
where the left square is a pullback square. This map is coCartesian just in case when $X\rightarrow X_2$ is an equivalence.

The fiber $\TopG_{[G/H]}$ is equivalent to $N(\Top^G_{/(G/H)})$ which is further equivalent to $N(\Top^H)$ by taking the fiber over $eH$.

\subsection*{Example: $G$-$\infty$-category of finite $G$-sets}

The $G$-$\infty$-category of finite $G$-sets is a coCartesian fibration $\FinG \rightarrow \OG^{\op}$ where $\FinG$ is an $\infty$-category whose objects are written in the form $U\rightarrow O$, where $O$ is an orbit space and where $U$ is a finite $G$-set i.e. $U\in F_G$ where $F_G$ is a finite coproduct completion of transitive $G$-spaces with finite stabilizers. Morphism between two elements of $\FinG$, $U_1\rightarrow O_1$ and $U_2\rightarrow O_2$ can be written in a form of a commutative diagram
\begin{center}
\begin{tikzcd} [row sep=2em, column sep=2em]
U_1\arrow[d] & U \arrow[l]\arrow[d]\arrow[r] & U_2\arrow[d] \\
O_1  & O_2 \arrow[r,"="]\arrow[l] & O_2
\end{tikzcd}
\end{center}
with the left square being a summand inclusion i.e. $O_2 \times_{O_1} U_1 \simeq U \coprod U'$, where $U'$ is a finite $G$-set.

We will say that a morphism in $\FinG$ is \textbf{inert} if the map $U\rightarrow U_2$ is an equivalence, and that it is \textbf{active} if the map $U \rightarrow O_2 \times_{O_1} U_1$ is an equivalence.

\subsection*{Example: $G$-$\infty$-category of $G$-manifolds}

The $G$-$\infty$-category of $G$-manifolds is a coCartesian fibration $\MG \rightarrow \OG^{\op}$ with $\MG$ being the $\infty$-category whose objects are written in the form $p: M\rightarrow O$ where $O$ is an orbit space and where $p$ is an equivariant manifold-bundle map whose fiber is a smooth equivariant manifold of a fixed dimension $n$. The idea of having such objects is in the following: by taking a basepoint $x$ of $O$ we can write $(O,x)\cong G/H$ where $H$ is the stabilizer of $x$. Now $p:M\rightarrow G/H$ encodes an $H$-manifold by taking the fiber of $p$ over $eH$.

A morphism between two objects $M_1\rightarrow O_1$ and $M_2\rightarrow O_2$ can be written in a form of a commutative diagram
\begin{center}
\begin{tikzcd} [row sep=2em, column sep=2em]
M_1 \arrow[d] & M \arrow[l]\arrow[r]\arrow[d] &  \arrow[d] M_2 \\
O_1 & O_2\arrow[l] \arrow[r,"="] & O_2
\end{tikzcd}
\end{center}
such that the left square is a pullback square. This morphism is coCartesian just in case the right square is a $G$-isotopy equivalence. Additionally, in any square in the upper diagram of the form
\begin{center}
\begin{tikzcd} [row sep=2em, column sep=2em]
M_1 \arrow[d]\arrow[r] &  \arrow[d] M_2 \\
O_1 \arrow[r] & O_2
\end{tikzcd}
\end{center}
The induced map $M_1\rightarrow O_1\times_{O_2} M_2$ is fiberwise a (smooth) $G$-embedding map.

The benefit of having such $G$-$\infty$-category is that it is relatively easy to describe the functors of restriction and topological induction. Let $K\leq H\leq G$ be finite subgroups. The restriction functor
\[
\Res^H_K : \Mfld^H\simeq \MG_{[G/H]} \rightarrow \MG_{[G/K]}\simeq \Mfld^K
\]
sends $M\rightarrow G/H$ to $N\rightarrow G/K$ where the manifold $N$ is obtained as a pullback of the following diagram
\begin{center}
\begin{tikzcd} [row sep=2em, column sep=2em]
M \arrow[d] & N \arrow[l]\arrow[d] \\
G/H & G/K\arrow[l]
\end{tikzcd}
\end{center}
Let $M_H$ be the fiber of $M\rightarrow G/H$ over $eH$ and let $N_K$ be the fiber of $N\rightarrow G/K$ over $eK$. Forgetting the action, we have $M_H\cong N_K$ since the upper diagram is a pullback square. Additionally, $M_H$ is equipped with $H$-action while $N_K$ is equipped with $K$-action. Since all of the maps in the upper diagram are $G$-equivariant, we see that $N_K$ is equivalent to $M_H$ with restricted $K$-action.

To add on, the topological induction $H \times_K N_K$ (where $N_K$ is the fiber of $N\rightarrow G/K$ over $eK$) is simply given by post composition $N\rightarrow G/K\rightarrow G/H$.

\subsubsection{$G$-symmetric monoidal $\infty$-categories}

\begin{definition}
\label{gsymdef}
(see \cite[section 3.1]{DN17} or \cite[B.0.10]{AH19}) A $G$-symmetric monoidal $G$-$\infty$-category is a coCartesian fibration $\Cs\rightarrow\FinG$ such that for every finite pointed $G$-set $J = (U\rightarrow O)\in\FinG$ we have an equivalence
\[
\underset{W\in \Orbit(U)}{\prod} (\chi_{[W\subset U]})_{!} : \Cs_{[J]} \rightarrow \underset{W\in \Orbit(U)}{\prod} \Cs_{[I(W)]}
\]
where $I(W) = (W\xrightarrow{=}W)\in\FinG$, and where $\Cs_{[J]}$ and $\Cs_{[I(W)]}$ are the corresponding fibers. The functor $(\chi_{[W\subset U]})_! : \Cs_{[J]} \rightarrow \Cs_{[I(W)]}$ is a pushforward functor associated to the following inert map in $\FinG$:
\begin{center}
\begin{tikzcd} [row sep=2em, column sep=2em]
U\arrow[d] &  W \arrow[l]\arrow[r,"="]\arrow[d,"="] & W \arrow[d,"="] \\
O & W\arrow[l]\arrow[r,"="] & W
\end{tikzcd}
\end{center}
\end{definition}

\begin{remark}
\label{undergsym}
Given a $G$-symmetric monoidal $G$-$\infty$-category $\Cs\rightarrow\FinG$ we can define its underlying $G$-$\infty$-category $\Cic := \Cs_{I(-)}$ which fits into the pullback square
\begin{center}
\begin{tikzcd} [row sep=2em, column sep=2em]
\Cs_{I(-)} \arrow[d]\arrow[r] &  \Cs \arrow[d] \\
\OG^{\op} \arrow[r,"I(-)"] & \FinG
\end{tikzcd}
\end{center}
\noindent where $I: \OG\rightarrow\FinG$ is the functor from \ref{gsymdef}.
\end{remark}

\subsection*{Example: $G$-symmetric monoidal category of $G$-manifolds}

We can equip the $G$-$\infty$-category of $G$-manifolds with a $G$-symmetric monoidal structure by considering the disjoint union of $G$-manifolds. Given two objects $M_1\rightarrow O_1$ and $M_2\rightarrow O_2$, their disjoint union can be written as $M_1 \sqcup M_2 \rightarrow O_1 \coprod O_2$, hence the motivation for the following definition.

The $G$-symmetric monoidal category of $G$-manifolds is a coCartesian fibration $\MGs\rightarrow\FinG$ with $\MGs$ being an $\infty$-category whose objects are equivariant manifold-bundle maps in the form $M\rightarrow U\rightarrow O$ where $U$ is a finite $G$-set. A map between two such objects $M_1\rightarrow U_1\rightarrow O_1$ and $M_2\rightarrow U_2\rightarrow O_2$ can be written in a form of a commutative diagram
\begin{center}
\begin{tikzcd} [row sep=3em, column sep=3em]
M_1 \arrow[d] & M \arrow[r]\arrow[l]\arrow[d] & M_2 \arrow[d] \\
U_1 \arrow[d] & U \arrow[r]\arrow[l]\arrow[d] & U_2 \arrow[d] \\
O_1 & O_2 \arrow[l]\arrow[r,"="] & O_2
\end{tikzcd}
\end{center}
where the left side is equivalent to a pullback over a summand inclusion. This morphism is coCartesian just in case when the right side is equivalent to the identity of manifolds. The underlying $G$-$\infty$-category of $\MGs$ will be $\MG$.

\subsubsection{$G$-$\infty$-operads}

\begin{definition}
\label{defgop}
(\cite[def. 3.1]{DN17}) A $G$-$\infty$-operad is an inner fibration $p:E^{\otimes}\rightarrow \OG^{\op}$ satisfying the following conditions:
\begin{itemize}
\item For every inert edge $e:J_1\rightarrow J_2$ in $\FinG$ and every $x\in E^{\otimes}_{[J_1]}$ there exists a coCartesian lift $\widetilde{e}:x\rightarrow y$ over $e$.

\item For any $J=[U\rightarrow O] \in\FinG$ and any choice of pushforward functors along inert edges, we have an equivalence
\[
\underset{W\in Orbit(U)}{\prod} (\chi_{[W\subseteq U]})_{!} : E^{\otimes}_{[J]} \xrightarrow{\simeq} \underset{W\in Orbit(U)}{\prod} E^{\otimes}_{I(W)}
\]

\item For any choice of pushforward functors along inert morphisms and for every map $e:J_1=[U_1\rightarrow O_1]\rightarrow J_2=[U_2\rightarrow O_2]$ in $\FinG$ and every $x\in O^{\otimes}_{[J_1]}$ and $y\in O^{\otimes}_{[J_2]}$ the map
\[
\Map^e_{E^{\otimes}} (x,y) \rightarrow \underset{W\in \Orbit(U_2)}{\prod} \Map^{\chi_{[W\subseteq U]}\circ e}_{E^{\otimes}} (x,(\chi_{[W\subseteq U]})_{!} y)
\]
is an equivalence where $\Map^e_{E^{\otimes}} (x,y)$ is the fiber over $e$ of the map $\Map_{E^{\otimes}} (x,y) \rightarrow \Map_{\FinG} (J_1,J_2)$.
\end{itemize}
An arrow $f:x\rightarrow y$ of a $G$-$\infty$-operad $p:E^{\otimes}\rightarrow\FinG$ is called \textbf{inert} if it is $p$-coCartesian and if $p(f)$ is inert in $\FinG$. Let $p:E^{\otimes}\rightarrow\FinG$ and $q:E'^{\otimes}\rightarrow\FinG$ be two $G$-$\infty$-operads and let $F:E^{\otimes}\rightarrow E'^{\otimes}$ be a map over $\FinG$. We say that $F$ is a map of $G$-$\infty$-operads if $F$ sends inert edges to inert edges.
\end{definition}

\begin{remark}
Every $G$-symmetric monoidal $\infty$-category is a $G$-$\infty$-operad. Moreover, just as in \ref{undergsym} we can define the underlying $G$-$\infty$-category of a $G$-$\infty$-operad to be the pullback
\begin{center}
\begin{tikzcd} [row sep=2em, column sep=2em]
E^{\otimes}_{I(-)} \arrow[d]\arrow[r] &  E^{\otimes} \arrow[d] \\
\OG^{op} \arrow[r,"I(-)"] & \FinG
\end{tikzcd}
\end{center}
We will usually mark with $E$ the underlying $G$-$\infty$-category ${E^{\otimes}}_{I(-)}$.
\end{remark}

\begin{example}
Every $G$-symmetric monoidal category is a $G$-$\infty$-operad.
\end{example}

\begin{definition}
Let $p:E^{\otimes}\rightarrow\FinG$ and $q:E'^{\otimes}\rightarrow\FinG$ be two $G$-$\infty$-operads and let $F:E^{\otimes}\rightarrow E'^{\otimes}$ be a map of $G$-$\infty$-operads. We will call $F$ an $E^{\otimes}$-algebra in $E'^{\otimes}$. Additionally, we will mark with $\Alg_{G,E^{\otimes}}(E'^{\otimes})$ or $\Alg_G (E^{\otimes},E'^{\otimes})$ the $\infty$-category of $E^{\otimes}$-algebras in $E'^{\otimes}$.
\end{definition}

\subsection*{Example: $G$-representations}

Consider $\RG$ the full $G$-$\infty$-subcategory of $\MG$ spanned by the $G$-vector bundle maps $E\rightarrow O$. Note that one such objects of the form $E\rightarrow G/H$ encodes an $H$-representation by taking the fiber over $eH$.

Now let $\RGs$ be the full $G$-$\infty$-subcategory of $\MGs$ on objects of $\RG$ i.e. the objects of $\RGs$ will be of the form $E\rightarrow U\rightarrow O$ where $E\rightarrow U$ is a $G$-vector bundle, with $U$ being a finite $G$-set (more on this will be said in \ref{gdiscs}). The map $\RGs\rightarrow\FinG$ is a $G$-$\infty$-operad.

The categories $\RG$ and $\RGs$ will be closely related to $G$-discs as we shall see in the following section.

\addtocontents{toc}{\protect\setcounter{tocdepth}{2}}

\subsection{$G$-discs}
\label{gdiscs}
Consider a map $U\rightarrow O$ where $O\in \OG$ and where $U$ is a finite $G$-set. Note that this map is a covering map. Therefore for a $G$-vector bundle $E\rightarrow U$ the composite map $E\rightarrow U\rightarrow O$ is a $G$-manifold bundle, hence, an element of $\MG$.
\begin{definition}
Define a $G$-disc to be a vector bundle $E\rightarrow O$ of rank $n$ where $O\in\OG$. Denote with $\DG\subset\MG$ the full $G$-$\infty$-subcategory spanned by objects of the form $E\rightarrow U\rightarrow O$ where $U$ is a finite $G$-set and $E\rightarrow U$ is a $G$-vector bundle.
\end{definition}
Similar to $G$-manifolds the $G$-$\infty$-category of $G$-discs can be endowed with the $G$-symmetric monoidal structure
\begin{definition}
Let $\DGs \subset \MGs$ be the full subcategory spanned by those elements  equivalent to $E\rightarrow U\rightarrow V \rightarrow O$ where $U, V$ are finite $G$-sets, $O\in\OG$ and where $E\rightarrow U$ is a $G$-vector bundle. We will say that $\DGs$ is the $G$-symmetric monoidal category of $G$-discs.
\end{definition}

\addtocontents{toc}{\protect\setcounter{tocdepth}{1}}

\subsubsection{Framing structure}

By a framing structure on a $G$-manifold we mean a tangential structure induced by a framing map. To be more precise, let $B$ be a $G$-space and let $f: B\rightarrow \BOG$ be a $G$-map. The $f$-framing (or $B$-framing) on a $G$-manifold $M$ is given by a $G$-map $f_M: M\rightarrow B$ together with a $G$-homotopy commutative diagram
\begin{center}
\begin{tikzcd} [row sep=2em, column sep=2em]
 & B \arrow[dr,"f"] & \\
M \arrow[ur,"f_M"]\arrow[rr,"\tau_M"] &  & \BOG
\end{tikzcd}
\end{center}
where $M\xrightarrow{\tau_M}\BOG$ is the tangent bundle classifier map. The most important example is when $B$ is taken to be a point, $B = \ast$. This framing corresponds to the diagram
\begin{center}
\begin{tikzcd} [row sep=2em, column sep=2em]
TM \arrow[r]\arrow[d] & V \arrow[d]\arrow[r] & \textrm{EO}_n(G) \arrow[d] \\
M \arrow[r,"f_M"] & \ast \arrow[r,] & \BOG
\end{tikzcd}
\end{center}
where $V$ is a $G$-representation. Hence, this $V$-framing, corresponds to the trivialization of the tangent vector bundle $TM\cong V\times M$.

One can obtain a $G$-$\infty$-category of $f$-framed $G$-manifolds $\MGB$ (see \cite{AH19} or \cite{Mil22}), from which we define the $G$-$\infty$-category of $f$-framed $G$-discs $\DGB$ as a pullback
\begin{center}
\begin{tikzcd} [row sep=3em, column sep=3em]
\DGB \arrow[r]\arrow[d] & \MGB \arrow[d] \\
\DG \arrow[r,hookrightarrow] & \MG
\end{tikzcd}
\end{center}
Additionally, both categories $\MGB$ and $\DGB$ can be endowed with the $G$-symmetric monoidal structure, which we will mark as $\MGBs$ and $\DGBs$ respectively.

\subsubsection{Description of $G$-discs}
\label{framingongdiscs}
By definition of a $G$-disc $E\rightarrow U\rightarrow O$, the map $E\rightarrow U$ is a $G$-vector bundle, hence, we can write this $G$-disc in the form
\begin{center}
\begin{tikzcd} [row sep=2em, column sep=3em]
U \arrow[r]\arrow[d] & \BOG \\
O &
\end{tikzcd}
\end{center}
where the horizontal map is the bundle classifying map. Consequently, we can think of an $f$-framed $G$-disc (where $f:B\rightarrow\BOG$ is a $G$-map) in a form of
\begin{center}
\begin{tikzcd} [row sep=2em, column sep=3em]
U \arrow[r]\arrow[d] & B \\
O &
\end{tikzcd}
\end{center}
where the horizontal map is now the framing map. Note that in both cases we can write these $G$-discs as $U\rightarrow O\times \BOG$ and $U\rightarrow O\times B$ respectively.
\begin{definition}
Let $\Cs$ be a $G$-symmetric monoidal $G$-$\infty$-category, let $B$ be a $G$-space and let $f:B\rightarrow\BOG$ be a $G$-map. Define the $\infty$-category of \emph{$f$-framed $G$-disc algebras with values in $\Cs$} to be the $\infty$-category of $G$-symmetric monoidal functors $\FunGs (\DGBs,\Cs)$.
\end{definition}
In the case when $B=\ast$, the map $\ast\rightarrow\BOG$ corresponds to the $n$-dimensional $G$-representation $V$. We will call $\FunGs (\underline{Disk}^{G,V-fr,\sqcup},\Cs)$ the $\infty$-category of $V$-framed $G$-disc algebras with coefficients in $\Cs$.

\subsubsection{$G$-monoidal envelope}

The most important property of the $G$-$\infty$-category of (framed) $G$-discs is that it is equivalent to the $G$-symmetric monoidal envelope of the $G$-$\infty$-operad $\RGBs$. The results can be summed up by the following proposition.
\begin{proposition}
\label{genv}
(\cite[3.7.4]{AH19} or \cite[7.2.5]{Mil22}) Let $B$ be a $G$-space and let $f:B\rightarrow\BOG$ be a $G$-map. Then the $G$-symmetric monoidal $\infty$-category of $f$-framed $G$-discs $\DGBs$ is equivalent to $\Env_G (\RGBs)$, the $G$-symmetric monoidal envelope of $\RGBs$.
\end{proposition}
\begin{corollary}
(\cite[7.2.6]{Mil22}) Let $\Cs$ be a $G$-symmetric monoidal $\infty$-category, let $B$ be a $G$-space and let $f:B\rightarrow\BOG$ be a $G$-map. Then there is an equivalence of $\infty$-categories
\[
\FunGs (\DGBs,\Cs)\xrightarrow{\simeq} \Alg_G (\RGBs,\Cs)
\]
\end{corollary}

\addtocontents{toc}{\protect\setcounter{tocdepth}{2}}

\section{Universal property of framed $G$-disc algebras}

\subsection{$G$-approximations to $G$-$\infty$-operads}

In this section we will present the theory of $G$-approximations to $G$-$\infty$-operads. Informally, a $G$-approximation to a $G$-$\infty$-operad is an $\infty$-category that somewhat behaves as a $G$-$\infty$-operad. It is a more general object that can capture information about the original $G$-$\infty$-operad.

\begin{definition}
\label{defgapprox}
Let $p:E^{\otimes}\rightarrow\FinG$ be a $G$-$\infty$-operad and let $f:C\rightarrow E^{\otimes}$ be a categorical fibration. We say that $f$ is a \emph{$G$-approximation} to $E^{\otimes}$ if it satisfies the following conditions:
\begin{enumerate}
\item Let $p'=p\circ f$. The $\infty$-category $C_{I(-)}$ obtained as the pullback
\begin{center}
\begin{tikzcd} [row sep=2em, column sep=2em]
C_{I(-)} \arrow[d]\arrow[r] & C \arrow[d,"p' "] \\
\OG^{\op} \arrow[r] & \FinG
\end{tikzcd}
\end{center}
together with the map $C_{I(-)}\rightarrow \OG^{\op}$ is a $G$-$\infty$-category i.e. this map is a coCartesian fibration.

\item Let $c\in C$ and let $p'(c) = I = [U\rightarrow O]$. Then there exists a $p'$-coCartesian morphism $c\rightarrow c_W$ in $C$ lifting $\chi_{W\subseteq U}: I\rightarrow I(W)$ for every $W\in \Orbit(U)$. Additionally, the map $f(c\rightarrow c_W)$ is an inert map in $E^{\otimes}$.

\item Let $c\in C$ and let $\alpha: x\rightarrow f(c)$ be an active morphism in $E^{\otimes}$. Then there exists a Cartesian lift $\bar{\alpha}: \bar{x}\rightarrow c$ of $\alpha$ in $C$.
\end{enumerate}
\end{definition}

\begin{definition}
Let $p:E^{\otimes}\rightarrow\FinG$ and $q:E'^{\otimes}\rightarrow\FinG$ be two $G$-$\infty$-operads and let $f:C\rightarrow E^{\otimes}$ be a $G$-approximation. Denote with $p' = p\circ f$ and consider a functor $F:C\rightarrow E'^{\otimes}$. We will say that $F$ is a \emph{$C$-algebra object} in $E'^{\otimes}$ if the following conditions are satisfied:
\begin{itemize}
\item $F$ induces a $G$-functor between the underlying $G$-$\infty$-categories $C_{I(-)}\rightarrow E'$.
\item The following diagram is commutative
\begin{center}
\begin{tikzcd} [row sep=2em, column sep=2em]
C \arrow[d,"f"]\arrow[r,"F"] & E'^{\otimes} \arrow[d,"q"] \\
E^{\otimes} \arrow[r,"p"] & \FinG
\end{tikzcd}
\end{center}
\item Let $x\in C$ and let $p'(x) = [U\rightarrow O]$. Let $\alpha_W : x\rightarrow x_W$ be a $p'$-coCartesian arrow lifting $\chi: [U\rightarrow O] \rightarrow [W\xrightarrow{=}W]$ for $W\in \Orbit(U)$. Then the map $F(\alpha_W)$ is an inert map in $E'^{\otimes}$.
\end{itemize}
We will denote with $\Alg_G(C,E')$ the $\infty$-category of $C$-algebra objects in $E'^{\otimes}$.
\end{definition}
We would like to prove the proposition \ref{ha23323} (the equivariant version of \cite[2.3.3.23]{HA}). The following lemmas will be helpful:
\begin{lemma}
\label{gopinert}
Let $p:E^{\otimes}\rightarrow\FinG$ and $q:E'^{\otimes}\rightarrow\FinG$ be two $G$-$\infty$-operads and let $F:E^{\otimes}\rightarrow E'^{\otimes}$ be a map over $\FinG$. Let $\mathcal{F}$ be a class of arrows in $\FinG$ of type
\begin{center}
\begin{tikzcd} [row sep=2em, column sep=2em]
U \arrow[d] & O_2 \arrow[l]\arrow[r,"="]\arrow[d,"="] & O_2 \arrow[d,"="] \\
O_1 & O_2 \arrow[l]\arrow[r,"="] & O_2
\end{tikzcd}
\end{center}
Then $F$ is a map of $G$-$\infty$-operads if and only if it sends an inert arrow in $E^{\otimes}$ over $\mathcal{F}$ to an inert arrow in $E'^{\otimes}$.
\end{lemma}
\begin{proof}
For a finite $G$-set $I=[U\rightarrow O]$ note that the maps $\chi_{[W\subseteq U]}$ all belong to $\mathcal{F}$. The proof is now analogous to \cite[2.1.2.9]{HA}.
\end{proof}

\begin{lemma}
\label{weakgapprox}
Let $p:E^{\otimes}\rightarrow\FinG$ be a $G$-$\infty$-operad and let $f:C\rightarrow E^{\otimes}$ be a $G$-approximation. Let $c\in C$ and let $\alpha: x\rightarrow f(c)$ be any morphism in $E^{\otimes}$. Consider
\[
\Sigma\subseteq C_{/c} \times_{E^{\otimes}_{/f(c)}} {E^{\otimes}_{x/}}_{/f(c)}
\]
the full subcategory spanned by the objects corresponding to pairs $\{\beta:d\rightarrow c, \gamma: x\rightarrow f(d)\}$ such that $\gamma$ is inert. The $\infty$-category $\Sigma$ is contractible.
\end{lemma}
\begin{proof}
We can factorize $\alpha$ as in the following commutative diagram
\begin{center}
\begin{tikzcd} [row sep=2em, column sep=2em]
 & y \arrow[rd,"\alpha_2"] & \\
x \arrow[ur,"\alpha_1"]\arrow[rr,"\alpha"] & & f(c)
\end{tikzcd}
\end{center}
such that $\alpha_1$ is inert and $\alpha_2$ is active. By definition of a $G$-approximation, we have a Cartesian lift $\bar{\alpha}_2: \bar{y}\rightarrow c$ of $\alpha_2$ in $C$. Now, analogous to the proof of \cite[2.3.3.10]{HA} we can claim that the pair $\sigma = (\bar{\alpha}_2: \bar{y}\rightarrow c,\alpha_1 : x\rightarrow y = f(\bar{y}))$ is a terminal object of $\Sigma$, hence $\Sigma$ is contractible.
\end{proof}

\begin{remark}
In Lurie's book \cite{HA}, one can find a definition of a \emph{weak approximation} to be a map $C\rightarrow E^{\otimes}$ satisfying the second point from the definition \ref{defgapprox} and the condition from \ref{weakgapprox} (in the non-equivariant case). In this paper, we will stick to the work with (strong) $G$-approximation although one could add the definition of the weak one.
\end{remark}

\begin{proposition}
\label{ha23323}
Let $p:E^{\otimes}\rightarrow\FinG$ and $q:E'^{\otimes}\rightarrow\FinG$ be two $G$-$\infty$-operads and let $f:C\rightarrow E^{\otimes}$ be a $G$-approximation. Assume that the induced map $C_{I(-)}\rightarrow E$ is an equivalence of $G$-$\infty$-categories. Then the induced map
\[
\theta: \Alg_G(E,E') \rightarrow \Alg_G(C,E')
\]
is an equivalence of $\infty$-categories.
\end{proposition}
\begin{proof}
First, we can choose a Cartesian fibration $u:M\rightarrow \Delta^1$ associated to the functor $f$, together with isomorphisms $M\times_{\Delta^1} \{0\}\simeq E^{\otimes}$, $M\times_{\Delta^1} \{1\}\simeq C$ and a retraction $r:M\rightarrow E^{\otimes}$ such that $r|_{C} = f$. Denote with $\Upsilon \subseteq \Fun_{\FinG} (M,E')$ the full subcategory spanned by the functors $F:M\rightarrow E'$ over $\FinG$ such that:
\begin{enumerate}
\item The restriction $F|_{E^{\otimes}}$ belongs to $\Alg_G(E,E')$,
\item For every $u$-Cartesian morphism $\alpha$ of $M$, the image $F(\alpha)$ is an equivalence in $E'$ (or, equivalently, $F$ is a $q$-left Kan extension of $F|_{E^{\otimes}}$).
\end{enumerate}
We will continue the proof and complete it in several steps.
\begin{step}
The restriction map $\Upsilon\rightarrow \Alg_G(E,E')$ is a trivial Kan fibration.
\end{step}
Let $\mathcal{X}\subseteq \Fun_{\FinG}(M,E'^{\otimes})$ be the full subcategory spanned by those functors $F:M\rightarrow E'^{\otimes}$ which are $q$-left Kan extensions of $F|_{E^{\otimes}}$, and let $\mathcal{Y}\subseteq \Fun_{\FinG} (E^{\otimes},E'^{\otimes})$ be the full subcategory spanned by those functors $F:E^{\otimes}\rightarrow E'^{\otimes}$ such that for every $x\in M$ the induced functor $E^{\otimes}_{/x}\rightarrow E'^{\otimes}$ has a $q$-colimit, where $E^{\otimes}_{/x} := E^{\otimes}\times_M M_{/x}$. By \cite[4.3.2.15]{HTT} the restriction functor $\mathcal{X}\rightarrow\mathcal{Y}$ is a trivial Kan fibration. Moreover, we claim that $\mathcal{Y}\simeq \Fun_{\FinG} (E^{\otimes},E'^{\otimes})$. This follows from the fact that $E^{\otimes}_{/x}$ has a terminal object, hence every functor in $\Fun_{\FinG} (E^{\otimes},E'^{\otimes})$ has a $q$-colimit. The terminal object in $E^{\otimes}_{/x}$ is given by $x\rightarrow x$ if $x\in M|_{\{0\}}\simeq E^{\otimes}$ and by a Cartesian lift $y\rightarrow x$ over $x$ if $x\in M|_{\{1\}}\simeq C$. By the description of the $\infty$-category $\Upsilon$ the restriction functor $\Upsilon\rightarrow \Alg_G(E,E')$ fits in the commutative diagram
\begin{center}
\begin{tikzcd} [row sep=2em, column sep=2em]
\Upsilon \arrow[d]\arrow[r,hookrightarrow] & \mathcal{X} \arrow[d]\arrow[r] & \Fun_{\FinG}(M,E') \arrow[d] \\
\Alg_G(E,E') \arrow[r,hookrightarrow] & \mathcal{Y} \arrow[r,"\simeq"] & \Fun_{\FinG} (E,E')
\end{tikzcd}
\end{center}
In fact, the left square is a pullback square, again by the description of $\Upsilon$, hence $\Upsilon\rightarrow \Alg_G(E,E')$ is a trivial Kan fibration.

Going toward our next step, note that precomposition with $r$ gives a section of this trivial fibration, which we will traditionally mark with $s$. Let $\epsilon: \Upsilon\rightarrow \Fun_{\FinG}(C,E'^{\otimes})$ be the other restriction. The functor $\theta$ is given by the composition $\epsilon\circ s$, which means that it will suffice to prove that $\epsilon$ induces an equivalence of $\infty$-categories $\Upsilon$ and $\Alg_G(C,E'^{\otimes})$.

Again, by \cite[4.3.2.15]{HTT} it will suffice to show the following:
\begin{enumerate}
\item for every $F_0\in \Alg_G(C,E')$ there exists $F\in \Fun_{\FinG}(M,E'^{\otimes})$ such that $F$ is a $q$-right Kan extension of $F_0$;
\item a functor $F\in \Fun_{\FinG}(M,E'^{\otimes})$ belongs to $\Upsilon$ if and only if it is a $q$-right Kan extension of $F_0 = F|_C$, with $F_0\in \Alg_G(C,E')$.
\end{enumerate}
Naturally, our next step would be to prove:
\begin{step}
\label{step2}
For every $F_0\in \Alg_G(C,E')$ there exists $F\in \Fun_{\FinG}(M,E'^{\otimes})$ such that $F$ is a $q$-right Kan extension of $F_0$.
\end{step}
Take $x\in E^{\otimes}$, let $C_{x/} = M_{x/}\times_M C$ and let $F_x = F_0 |_{C_{x/}}$. By \cite[4.3.2.13]{HTT} it will suffice to show that $F_x$ can be extended to a $q$-limit diagram $C_{x/}^{\lhd}\rightarrow E'^{\otimes}$ (covering the map $C_{x/}^{\lhd}\rightarrow M\rightarrow \FinG$). Denote with $C'_{x/}$ the full subcategory of $C_{x/}$ spanned by those morphisms $x\rightarrow y$ in $M$ such that $x\rightarrow f(y)$ is an inert arrow in $E^{\otimes}$. By \ref{weakgapprox} the inclusion functor $C'_{x/}\hookrightarrow C_{x/}$ is final, thus by \cite[4.3.1.7]{HTT} it will suffice to show that $F'_{x} = F_{x}|_{C'_{x/}}$ can be extended to a $q$-limit ${C'}_{x/}^{\lhd}\rightarrow E'^{\otimes}$.

Let $p(x) = [U\rightarrow O]$ and let $C''_{x/}$ be the full subcategory of $C'_{x/}$ spanned by inert morphisms $x\rightarrow f(y)$ such that their underlying map in $\FinG$ can be written as a span
\begin{center}
\begin{tikzcd} [row sep=2em, column sep=2em]
U \arrow[d] & V \arrow[d,"="]\arrow[r,"="]\arrow[l] & V\arrow[d,"="] \\
O & V\arrow[r,"="]\arrow[l] & V
\end{tikzcd}
\end{center}
where $V\in\OG$. At the moment, we wish to prove that $F'_x$ is a $q$-right Kan extension of $F''_x = F'_x |_{C''_{x/}}$ so that we can use \cite[4.3.2.7]{HTT}. Let $\alpha : x\rightarrow y$ be a map in $M$ which is an object of $C'_{x/}$ and denote $p\circ f(y) = [U_1\rightarrow O_1]$. The $\infty$-category $C''_{x/} \times_{C'_{x/}} (C'_{x/})_{\alpha /}$ can be identified with the full subcategory of $M_{\alpha /}$ spanned by diagrams $x\xrightarrow{\alpha} y \xrightarrow{\beta} z$ such that $p\circ f(\beta)$ can be written as the span
\begin{center}
\begin{tikzcd} [row sep=2em, column sep=2em]
U_1 \arrow[d] & V \arrow[d,"="]\arrow[r,"="]\arrow[l] & V\arrow[d,"="] \\
O_1 & V\arrow[r,"="]\arrow[l] & V
\end{tikzcd}
\end{center}
with $V\in\OG$. Note that the upper left $G$-map $V\rightarrow U_1$ factors through some $W\in \Orbit(U_1)$. Let, $\mathcal{F}_{W\subseteq U_1}$ be the full subcategory of ${\FinG}_{\chi_{[W\subseteq U_1]}/}$ spanned by those object (or better said, morpshisms) of the form
\begin{center}
\begin{tikzcd} [row sep=2em, column sep=2em]
U_1 \arrow[d] & V \arrow[d,"="]\arrow[r,"="]\arrow[l] & V\arrow[d,"="] \\
O_1 & V\arrow[r,"="]\arrow[l] & V
\end{tikzcd}
\end{center}
with $V\in\OG$. To be more precise, one such morphism can be factored as composition of maps 
\begin{center}
\begin{tikzcd} [row sep=2em, column sep=2em]
U_1 \arrow[d] & W \arrow[d,"="]\arrow[r,"="]\arrow[l] & W\arrow[d,"="] & V \arrow[d,"="]\arrow[l]\arrow[r,"="] & V \arrow[d,"="] \\
O_1 & W\arrow[r,"="]\arrow[l] & W & V \arrow[l]\arrow[r,"="] & V
\end{tikzcd}
\end{center}
Now we can write our $\infty$-category $C''_{x/} \times_{C'_{x/}} (C'_{x/})_{\alpha /}$ as the disjoint union of $\infty$-categories $C''(W)_{y/}$, for $W\in \Orbit(U_1)$, where each $C''(W)_{y/}$ is equivalent to the full subcategory of $C_{y/}$ spanned by objects (that is, morphisms) covering a map from $\mathcal{F}_{W\subseteq U_1}$. Since $f$ is a $G$-approximation all of these $\infty$-categories $C''(W)_{y/}$ have an initial object given by the $p\circ f$-coCartesian lift $y\rightarrow y_W$ covering $\chi_{[W\subseteq U_1]}$. Hence, it will suffice to show that $F_0(y)$ is a $q$-product of the objects $\{F_0(y_W)\}_{W\in \Orbit(U_1)}$. Since $E'^{\otimes}$ is a $G$-$\infty$-operad it will suffice to show that the maps $F_0(y)\rightarrow F_0(y_W)$ are inert, which is true since $F_0\in \Alg_G(C,E)$.

We have shown that $F'_x$ is a $q$-right Kan extension of $F''_x = F'_x |_{C''_{x/}}$. By \cite[4.3.2.7]{HTT} it will suffice to prove that $F''_{x}$ can be extended to a $q$-limit diagram ${C''}_{x/}^{\lhd}\rightarrow E'^{\otimes}$ covering the map ${C''}_{x/}^{\lhd}\rightarrow M\rightarrow \FinG$. Similarly as before, let $C''(W)_{x/}$ (for $W\in \Orbit(U)$ with $p(x) = [U\rightarrow O]$) be the full subcategory of $C''_{x/}$ spanned by objects covering a map from $\mathcal{F}_{W\subseteq U}$. Now the $\infty$-category $C''_{x/}$ can be written as the disjoint union of $\infty$-categories $C''(W)_{x/}$. Denote with
\[
E(W) \subseteq E^{\otimes}_{x/} \times_{{\FinG}_{[U\rightarrow O]/}} \mathcal{F}_{W\subseteq U}
\]
the full subcategory such that we have an equivalence $C''(W)_{x/}\simeq E(W) \times_{E} C_{I(-)}$. Since $f$ induces an equivalence on the underlying $G$-$\infty$-categories $C_{I(-)}\simeq E$ this is possible to do. Furthermore, we can choose inert morphisms $x\rightarrow x_W$ as coCartesian lifts of $\chi_{[W\subseteq U]}$ which in turn represent the initial objects of $E(W)$ for every $W\in \Orbit(U)$. Since $f$ induces a categorical equivalence $C_{I(-)}\xrightarrow{\simeq} E$ we can write $x_W \simeq f(c_W)$ with $\alpha_W : x\rightarrow c_W$ an initial object of $C''(W)_{x/}$. Similarly to the previous part, we are required to prove the existence of a $q$-product of objects $F_0(c_W)$, with $W\in \Orbit(U)$, which follows from the fact that $E'^{\otimes}$ is a $G$-$\infty$-operad. This finishes the proof of the second step.

The arguments above give us the following equivalent condition for Step \ref{step2}:
\theoremstyle{thmstyleone}
\newtheorem*{hlemma}{Help lemma}
\begin{hlemma}
Let $F\in \Fun_{\FinG}(M,E'^{\otimes})$ be a functor such that $F_0 = F|_{C} \in \Alg_G (C,E')$. Let $x\in E^{\otimes}$ be an element such that $p(x) = [U\rightarrow O]$, and let $c_W\in C$ be the elements such that $f\circ p(c_W) = [W\xrightarrow{=}W]$ for every $W\in \Orbit(U)$. Additionally, let $\alpha_W : x\rightarrow c_W$ be maps in $M$ covering $\chi_{[W\subseteq U]}$ constructed as above. Then $F$ is a $q$-right Kan extension of $F_0$ if and only if $F(\alpha_W)$ are inert arrows in $E'^{\otimes}$ for every $W\in \Orbit(U)$.
\end{hlemma}
This help lemma will turn fundamental in the proof of our third and final step:
\begin{step}
A functor $F\in \Fun_{\FinG}(M,E'^{\otimes})$ belongs to $\Upsilon$ if and only if it is a $q$-right Kan extension of $F_0 = F|_C$, with $F_0\in \Alg_G(C,E')$.
\end{step}
Let $F\in\Upsilon$. The functor $F_0 = F|_C$ is equivalent to the functor $F|_{E^{\otimes}} \circ f$ hence $F_0 \in \Alg_G (C,E')$.  Our help lemma now implies that $F$ is a $q$-right Kan extension of $F_0$.

For the other direction, let $F\in \Fun_{\FinG}(M,E'^{\otimes})$ and let $F_0 = F|_{C}\in \Alg_G(C,E')$. Assume that $F$ is a $q$-right Kan extension of $F_0$. Let $c\in C$ with $f\circ p(c) = [U\rightarrow O]$. Furthermore, let $x=f(c)$. We would like to show that $F(x)\rightarrow F(c)$ in an equivalence in $E'^{\otimes}$. For that, let us choose $p'$-coCartesian (with $p' = f\circ p$) lifts $c\rightarrow c_W$ over $\chi_{[W\subseteq U]}$. Since $F_0 \in \Alg_G (C,E')$ by assumption, the maps $F_0(c)\rightarrow F_0(c_W)$ are all inert. Since $E'^{\otimes}$ is a $G$-$\infty$-operad, it will suffice to show that $F(x)\rightarrow F(c_W)$ are all inert, which is true by the help lemma above.

What is left to show is that $F|_{E}\in \Alg_G(E,E')$. By \ref{gopinert} it would only suffice to check that the inert map $x\rightarrow x_V$ lying over
\begin{center}
\begin{tikzcd} [row sep=2em, column sep=2em]
U \arrow[d] & V \arrow[d,"="]\arrow[r,"="]\arrow[l] & V\arrow[d,"="] \\
O & V\arrow[r,"="]\arrow[l] & V
\end{tikzcd}
\end{center}
maps to an inert map, with $V\in\OG$. This map factors as $x\rightarrow x_W\rightarrow x_V$ lying over 
\begin{center}
\begin{tikzcd} [row sep=2em, column sep=2em]
U \arrow[d] & W \arrow[d,"="]\arrow[r,"="]\arrow[l] & W\arrow[d,"="] & V \arrow[d,"="]\arrow[l]\arrow[r,"="] & V \arrow[d,"="] \\
O & W\arrow[r,"="]\arrow[l] & W & V \arrow[l]\arrow[r,"="] & V
\end{tikzcd}
\end{center}
with $x\rightarrow x_W$ being inert and $W\in \Orbit(U)$. By the dual of \cite[2.4.1.7]{HTT} the map $x_W\rightarrow x_V$ is also inert. It is enough to show that $F(x)\rightarrow F(x_W)$ and $F(x_W)\rightarrow F(x_V)$ are inert.

Similarly as before, we can assume that $x_W = f(c_W)$ and $x_V = f(c_V)$, for some $c_W,c_V\in C$. Since the maps $F(x_W)\rightarrow F(c_W)$ and $F(x_V)\rightarrow F(c_V)$ are equivalences in $E'^{\otimes}$, for the map $F(x)\rightarrow F(x_W)$ to be inert it would suffice to show that the composite map $F(x)\rightarrow F(x_W)\rightarrow F(c_W)$ is inert, which follows from our \emph{help lemma}. As for $F(x_W)\rightarrow F(x_V)$, note that it is equivalent to $F(c_W)\rightarrow F(c_V)$ which is further equivalent to $F_0(c_W)\rightarrow F_0(c_V)$. Finally, since $F_0 \in \Alg_G (C,E')$ this map is inert and so is $F(x_W)\rightarrow F(x_V)$. With this the proof is finished.
\end{proof}
As an immediate consequence we have the following:
\begin{corollary}
Let $f:E^{\otimes}\rightarrow E'^{\otimes}$ be a map of $G$-$\infty$-operads. Furthermore, assume $f$ is a $G$-approximation map that induces an equivalence on the underlying $G$-$\infty$-categories. Then $f$ is an equivalence of $G$-$\infty$-operads.
\end{corollary}

\subsection{The universal property}

In this section we will prove the universal property of $G$-disc algebras. Let $H\leq G$ be a finite subgroup of a compact Lie group $G$. The statement that we want to prove is the following: the $G$-symmetric monoidal category of $G/H$-framed $G$-discs is freely generated by the $H$-symmetric monoidal category of $V$-framed $H$-discs, where $V$ is an $H$-representation. In other words, for a $G$-symmetric monoidal $\infty$-category $\Cs$ we have the following equivalence:
\[
\Fun^{\otimes}_G (\underline{\Disk}^{G,G/H\fr},\Cs) \simeq \Fun^{\otimes}_H (\underline{\Disk}^{H,V\fr},\underline{G/H}\underline{\times} \Cs)
\]
where $\underline{G/H}\underline{\times} \Cs$ is the underlying $H$-$\infty$-category of $\Cs$.

Note that by \ref{genv} $\underline{\Disk}^{G,G/H\fr}$ and $\underline{\Disk}^{H,V\fr}$ are equivalent to the $G$-symmetric monoidal envelopes $\Env_G (\RGHs)$ and $\Env_H (\RHVs)$ respectively. For simplicity, let us denote with $\DGH := \RGHs$ and $\DHV := \RHVs$. Now the equivalence above can be written as
\begin{equation}
\label{universalpg}
\Alg_G (\DGH, \Cic) \simeq \Alg_H (\DHV, \underline{G/H}\underline{\times} \Cic)
\end{equation}

\begin{mydescription}
\label{lp10rmk21}
Before continuing, let us give more insight into the objects of $\RHVs$ and $\RGHs$. Since the objects of $(\RHVs)_{I(-)}$ are framed over a point, they correspond to the $V$-framed $G$-discs. Informally, using \ref{framingongdiscs} we can depict these objects as
\begin{center}
\begin{tikzcd} [row sep=2em, column sep=3em]
H/K \times V \arrow[d]\arrow[r] & V \arrow[d] \\
H/K \arrow[r]\arrow[d,"="] & \ast \\
H/K &
\end{tikzcd}
\end{center}
where the lower horizontal arrow is the framing map and the square is a pullback square.

Generally speaking, the underlying $H$-$\infty$-category $\Cic_H$ of a $G$-$\infty$-category $\Cic$ is equivalent to $\underline{G/H}\underline{\times}\Ci$ where $\underline{G/H} = (\OG^{\op})_{(G/H)/}$, hence, a $K$-object in $\Cic_H$ i.e. an object over an orbit $H/K$ corresponds to an object over $G/K\rightarrow G/H$ in $\Cic$. Additionally, composing with the map $G/K\rightarrow G/H$ represents topological induction. Therefore, the $H$-disc above corresponds to the element
\begin{center}
\begin{tikzcd} [row sep=2em, column sep=3em]
E \arrow[d]\arrow[r] & G\times_H V \arrow[d] \\
G/K \arrow[r]\arrow[d,"="] & G/H \\
G/K \arrow[d] & \\
G/H &
\end{tikzcd}
\end{center}
such that the square in the diagram is a pullback square. Such element is a $G/H$-framed $G$-disc, with the framing map $G/H\rightarrow\BOG$ being adjoint to the framing map on $H$-discs $\ast\rightarrow \operatorname{BO}_n(H)$. Moreover, note that the framing map is the same as the map inducing topological induction, that is the composition of the lower two vertical arrows. Therefore we can write this element as $E \rightarrow G/K \xrightarrow{=} G/K \rightarrow G/H$ or even simpler as $G/K \xrightarrow{=} G/K \rightarrow G/H$ (since all the information on the bundle over $G/K$ is carried by the framing map). With this depiction the general object of $\RHVs$ can be written as
\[
U \rightarrow O \rightarrow G/H
\]
where $U$ is a finite $G$-set and $O\in\OG$ is an orbit, with the composition $U\rightarrow G/H$ representing the framing map as well.

Similarly, an object of $(\RGHs)_{I(-)}$ can be written in the form
\begin{center}
\begin{tikzcd} [row sep=2em, column sep=3em]
E \arrow[d]\arrow[r] & G\times_H V \arrow[d] \\
G/K \arrow[r]\arrow[d,"="] & G/H \\
G/K  & 
\end{tikzcd}
\end{center}
where the lower horizontal map is the framing map and the square is a pullback square. Again, by \ref{framingongdiscs} we can write this element as
\begin{center}
\begin{tikzcd} [row sep=2em, column sep=3em]
G/K \arrow[r]\arrow[d,"="] & G/H \\
G/K & 
\end{tikzcd}
\end{center}
with the horizontal map being the framing map. Therefore, the elements of $\RGHs$ can be written in the form
\[
O \leftarrow U \rightarrow G/H
\]
where $U$ is a finite $G$-set and $O\in\OG$ is an orbit, the left arrow is the structure map while the right arrow is the framing map.
\end{mydescription}

In order to prove the equivalence \eqref{universalpg} we would like to be able to use \ref{ha23323}.
\begin{construction}
Let us construct a map $\theta : \DHV\rightarrow\DGH$ as part of the commutative diagram
\begin{center}
\begin{tikzcd} [row sep=2em, column sep=2em]
\DHV \arrow[r,"\theta"]\arrow[d] & \DGH \arrow[d] \\
\FinH \arrow[r] & \FinG
\end{tikzcd}
\end{center}
Note that $\FinH$ is equivalent to $\underline{G/H}\underline{\times}\FinG$. Therefore the map $\FinH\rightarrow\FinG$ is a left fibration.

By \ref{lp10rmk21} we have a description of the objects of $\DHV$ and $\DGH$. The object $(U\rightarrow O\rightarrow G/H)\in\DHV$ is then sent by the map $\theta$ to the object $(O\leftarrow U\rightarrow O\rightarrow G/H)\in\DGH$.
\end{construction}

\begin{proposition}
\label{lp10thm23}
The map $\theta:\DHV\rightarrow\DGH$ is a $G$-approximation map. Moreover, $\theta$ induces an equivalence on the underlying $G$-$\infty$-categories.
\end{proposition}
\begin{proof}
In order to prove that $\theta$ is a $G$-approximation map we need to prove the three points from \ref{defgapprox}:
\begin{enumerate}
\item Consider the following commutative diagram
\begin{center}
\begin{tikzcd} [row sep=2em, column sep=2em]
\DHV \arrow[rrr,"\theta"]\arrow[dd] & & & \DGH \arrow[dd] &  \\
 & \RHV \arrow[ul]\arrow[dd] & & & \RGH \arrow[ul]\arrow[dd] \\
\FinH \arrow[rrr] & & & \FinG & \\
& \OH^{\op} \arrow[ul,"I(-)"]\arrow[rrr] & & & \OG^{\op} \arrow[ul,"I(-)"]
\end{tikzcd}
\end{center}
where the right rhombus is a pullback diagram. If we restrict our attention to the part of the diagram
\begin{center}
\begin{tikzcd} [row sep=2em, column sep=3em]
\underline{\Rep}_n^{V\fr}(H) \arrow[r]\arrow[d] & \DHV \arrow[d]\simeq \RHVs \\
\OH^{\op} \arrow[r,"I(-)"]\arrow[d] & \FinH \arrow[d] \\
\OG^{\op} \arrow[r,"I(-)"] & \FinG
\end{tikzcd}
\end{center}
we have that the upper square is a pullback square by definition. Additionally, by inspection, the lower square is a pullback square, meaning that the outer rectangle is a pullback diagram. Therefore, the pullback of $\DHV\rightarrow\FinH\rightarrow\FinG$ along $I(-):\OG^{\op}\rightarrow\FinG$ is $\RHV\rightarrow \OH^{\op} \rightarrow \OG^{\op}$. Additionally, the first map in the composition is a coCartesian fibration since $\RHV$ is an $H$-$\infty$-category and $\OH^{\op}\simeq\underline{G/H}\rightarrow \OG^{\op}$ is a left fibration, meaning that the composition is a coCartesian fibration i.e. $\RHV\rightarrow \OG^{\op}$ is a $G$-$\infty$-category. Additionally, since the right rhombus is a pullback diagram, $\theta$ induces a map that fits into the following commutative diagram
\begin{center}
\begin{tikzcd} [row sep=2em, column sep=3em]
\underline{\Rep}_n^{V\fr}(H) \arrow[r]\arrow[d] & \underline{\Rep}_n^{G/H\fr}(G) \arrow[d] \\
\underline{G/H}\simeq\OH^{\op} \arrow[r] & \OG^{\op}
\end{tikzcd}
\end{center}
\item Let $x\in\DHV$ be such that $\theta(x)$ lies over $I : U\rightarrow O \in \FinG$. Note that $I$ lies in an image of $\FinH\rightarrow\FinG$ i.e. $I$ is an image of an element $I': U'\rightarrow O'$ of $\FinH$. Since $\FinH\rightarrow\FinG$ is a left fibration (and, in particular, a coCartesian fibration), the maps $\chi_{[W\subseteq U]}$ in $\FinG$ have coCartesian lifts in $\FinH$ which are exactly maps $\chi_{[W'\subseteq U']}$. Now, considering that $\DHV\rightarrow\FinH$ is an $H$-$\infty$-operad, the coCartesian lifts of $\chi_{[W'\subseteq U']}$ exist with the source $x$. By the commutativity of the upper diagram, these lifts represent the desired coCartesian lifts $x\rightarrow x_W$ of $\chi_{[W\subseteq U]}$ in $\DHV$. It is clear that the maps $\theta(x\rightarrow x_W)$ are inert in $\DGH$.
\item Recall that, by \ref{lp10rmk21}, objects of $\DGH$ can be depicted as
\begin{center}
\begin{tikzcd} [row sep=2em, column sep=2em]
U \arrow[r]\arrow[d] & G/H \\
O &
\end{tikzcd}
\end{center}
with $U\in\FinG$ and $O\in\OG$ where the horizontal map corresponds to the framing map. Now let $x\in \DGH$ lying over $U_1\rightarrow O$, $c\in \DHV$ lying over $U_2\rightarrow O$ and let $x\rightarrow\theta(c)$ be an active arrow lying over diagram
\begin{center}
\begin{tikzcd} [row sep=2em, column sep=2em]
U_1 \arrow[r]\arrow[d]\arrow[rr, bend left] & U_2 \arrow[d] & G/H \arrow[d,"="] \\
O \arrow[r,"="] & O \arrow[r] & G/H
\end{tikzcd}
\end{center}
Note that a map in $\DGH$ should be represented by a span, but since the underlying arrow is fiberwise active, the left side of the span would be the identity, hence we omit it in the diagram. The element $\bar{x}$ of the Cartesian lift corresponds to the element $U_1 \rightarrow O \xrightarrow{=} O \rightarrow G/H$.
\end{enumerate}

What is left to show is that $\theta$ induces equivalence on the underlying $G$-$\infty$-categories. Let us denote with $\theta_{I(-)}:\DHV_{I(-)}\rightarrow \DGH_{I(-)}$ the induced functor. Note that the element of $\DGH_{I(-)}$ can be written as
\begin{center}
\begin{tikzcd} [row sep=2em, column sep=2em]
O \arrow[r]\arrow[d,"="] & G/H \\
O &
\end{tikzcd}
\end{center}
which is the same as $O\xrightarrow{=} O\rightarrow G/H$, an element of $\DHV_{I(-)}$, hence $\theta_{I(-)}$ is essentially surjective. As for the mapping spaces, recall that the mapping space in $\DHV_{I(-)}$ between $E_1 \rightarrow O_1 \xrightarrow{=} O_1 \rightarrow G/H$ and $E_2 \rightarrow O_2 \xrightarrow{=} O_2 \rightarrow G/H$ consists of the spans
\begin{center}
\begin{tikzcd} [row sep=2em, column sep=2em]
E_1 \arrow[d] & E \arrow[d]\arrow[l]\arrow[r] & E_2 \arrow[d] \\
O_1 \arrow[d,"="] & O_2 \arrow[l]\arrow[r,"="]\arrow[d,"="] & O_2 \arrow[d,"="] \\
O_1 \arrow[dr] & O_2 \arrow[l]\arrow[r,"="]\arrow[d] & O_2 \arrow[dl] \\
 & G/H &
\end{tikzcd}
\end{center}
which is the same as the span
\begin{center}
\begin{tikzcd} [row sep=2em, column sep=2em]
E_1 \arrow[d] & E \arrow[d]\arrow[l]\arrow[r] & E_2 \arrow[d] & \\
O_1 \arrow[d,"="]\arrow[drrr] & O_2 \arrow[l]\arrow[r,"="]\arrow[d,"="]\arrow[drr] & O_2 \arrow[d,"="]\arrow[dr] & \\
O_1 & O_2 \arrow[l]\arrow[r,"="] & O_2 & G/H \\
\end{tikzcd}
\end{center}
representing the map of the image of $\theta_{I(-)}$. We conclude that $\theta_{I(-)}$ is fully faithful and hence an equivalence which finishes the proof.
\end{proof}

Propositions \ref{lp10thm23} and \ref{ha23323} give us the following result:

\begin{corollary}
\label{lp10cor24}
Let $E^{\otimes}$ be a $G$-$\infty$-operad. The map $\theta$ induces an equivalence
\[
\Alg_G(\DGH,E) \xrightarrow{\simeq} \Alg_G (\DHV,E)
\]
\end{corollary}
We are one step away from proving our universal property:
\begin{theorem}
\label{universalthm}
Let $\Cs$ be a $G$-symmetric monoidal category. Then the $G$-symmetric monoidal category of $G/H$-framed $G$-discs is freely generated by the $H$-symmetric monoidal category of $V$-framed $H$-discs. In other words, there is an equivalence
\begin{equation*}
\Fun^{\otimes}_G (\underline{\Disk}^{G,G/H\fr},\Cs) \xrightarrow{\simeq} \Fun^{\otimes}_H (\underline{\Disk}^{H,V\fr},\Cs_H)
\end{equation*}
where $\Cs_H$ is the underlying $H$-symmetric monoidal category of $G$-symmetric monoidal category $\Cs$, and where $V$ is a $G$-representation.
\end{theorem}
\begin{proof}
As discussed in the beginning of this section, the above statement is equivalent to
\[
\Alg_G(\DGH,\Cic) \xrightarrow{\simeq} \Alg_H(\DHV,\Cic_H)
\]
By \ref{lp10cor24} it will suffice to find an equivalence
\[
\Alg_G (\DHV,\Cic) \xrightarrow{\simeq} \Alg_H(\DHV,\Cic_H)
\]
Recall that the $H$-symmetric monoidal category $\Cs_H$ is obtained via the commutative diagram
\begin{center}
\begin{tikzcd} [row sep=2em, column sep=3em]
\Cs_H \arrow[d]\arrow[r] & \Cs \arrow[d] \\
\FinH \arrow[d]\arrow[r] & \FinG \arrow[d] \\
\underline{G/H} \arrow[r] & \OG^{\op}
\end{tikzcd}
\end{center}
where both inner rectangles (and consequently also the outer rectangle) are pullbacks.

Let $F$ be a $\DHV$-algebra object in $\Cs$ i.e. $F\in \Alg_G(\DHV,\Cic)$, and consider the following diagram
\begin{center}
\begin{tikzcd} [row sep=2em, column sep=2em]
 & \Cs_H \arrow[dr]\arrow[dd] & \\
\DHV \arrow[rr,"F" near start]\arrow[ur,dashed,"F' "]\arrow[dr]\arrow[dd] & & \Cs \arrow[dd] \\
 & \FinH \arrow[dr] & \\
\DGH \arrow[rr] & & \FinG
\end{tikzcd}
\end{center}
By the universal property of the pullback, $F$ induces a functor $F' :\DHV\rightarrow\Cs_H$. We would like to show that $F' \in \Alg_H(\DHV,\Cs_H)$.

Let $\alpha_W$ be an inert map in $\DHV$ covering the map $\chi_{[W\subseteq U]}: [U\rightarrow O] \rightarrow [W\xrightarrow{=}W]$ in $\FinH$. Since $\FinH\rightarrow\FinG$ is a left fibration and $F\in \Alg_G(\DHV,\Cs)$, $F(\alpha_W)$ is inert in $\Cs$. Again, since $\FinH\rightarrow\FinG$ is a left fibration, and consequently $\Cs_H\rightarrow\Cs$, the pullback of $F(\alpha_W)$ is inert in $\Cs_H$, which is exactly $F'(\alpha_W)$, meaning $F'\in \Alg_H(\DHV,\Cs_H)$. Similarly, if we would have taken $F'\in \Alg_H(\DHV,\Cs_H)$ the induced functor $F:\DHV\rightarrow\Cs_H \rightarrow\Cs$ would lie in $\Alg_G(\DHV,\Cs)$. By the universality of the maps to the pullback we have 
\[
\Alg_G (\DHV,\Cic) \xrightarrow{\simeq} \Alg_H(\DHV,\Cic_H)
\]
and the proof is finished.
\end{proof}

\subsection{Applications of the universal property}
\label{applications}

In this section, we will give two applications of the universal property of $G$-discs: algebras with genuine involution and $O(2)$-genuine objects, and associative algebras and $S^1$-genuine objects.

\addtocontents{toc}{\protect\setcounter{tocdepth}{1}}

\subsection*{Algebras with genuine involution and $O(2)$-genuine objects}

Let $G = O(2)$ and regard $H=\Ztwo$ as a subgroup generated by a reflection. To add on, let $\underline{\Sp}^{O(2)}$ be the $O(2)$-$\infty$-category of $O(2)$-spectra. Let $V$ be the adjoint representation of $O(2)$. It is evident that $V$ is endowed with $O(2)$-action, which, when restricted to $\Ztwo$-action becomes $\mathbb{R}^{\sigma}$ where $\sigma$ is an one dimensional sign representation.

Let $A_H$ be the $\mathbb{R}^{\sigma}$-framed $\Ztwo$-disc algebra object in $\underline{\Sp}^{\Ztwo}$. Let us give more insight into the structure of $A_H$. The objects of $\underline{\Disk}^{\Ztwo,\mathbb{R}^{\sigma}\fr}$ are given by the finite disjoint unions of
\begin{itemize}
\item $\mathbb{R}^{\sigma}$, which corresponds to the element $\mathbb{R}^{\sigma}\rightarrow \Ztwo /\Ztwo$;
\item the restriction $\Res^{\Ztwo}_{\{e\}}(\mathbb{R}^{\sigma}) = \mathbb{R}^1$, which corresponds to the element $\sqcup_{\Ztwo} \mathbb{R}^1 \rightarrow \Ztwo/\{e\}$;
\item the topological induction $\sqcup_{\Ztwo} \mathbb{R}^1$ obtained as the element $\sqcup_{\Ztwo} \mathbb{R}^1 \rightarrow \Ztwo/\{e\}\rightarrow \Ztwo / \Ztwo$.
\end{itemize}
Therefore, we can write
\begin{itemize}
\item $\Res^{\Ztwo}_{\{e\}} A_H (\mathbb{R}^{\sigma}) = A_H(\mathbb{R}^1)$;
\item $A(\sqcup_{\Ztwo} \mathbb{R}^1) \simeq N^{\Ztwo}_{e} (A_H (\mathbb{R}^{1}))$.
\end{itemize}
Additionally, the equivariant embedding
\[
(\sqcup_{\Ztwo} \mathbb{R}^1)\sqcup \mathbb{R}^{\sigma} \hookrightarrow \mathbb{R}^{\sigma}
\]
induces a map $N^{\Ztwo}_{e} (A_H (\mathbb{R}^{1})) \otimes A(\mathbb{R}^{\sigma}) \rightarrow A(\mathbb{R}^{\sigma})$. Meaning that the $\Ztwo$-spectrum $A(\mathbb{R}^{\sigma})$ has a structure of a $N^{\Ztwo}_{e} (A_H (\mathbb{R}^{1}))$-module.

Next, note that $A_H(\mathbb{R}^1)$ has a structure of an $\mathbb{E}_1$-algebra object in $\underline{\Sp}^{\Ztwo}_{[\Ztwo/\{e\}]} \simeq \Sp$, since
\[
\Fun^{\otimes}(\underline{\Disk}^{\Ztwo,\mathbb{R}^{\sigma}\fr}_{[\Ztwo/\{e\}]} , \underline{\Sp}^{\Ztwo}_{[\Ztwo/\{e\}]}) \simeq \Fun^{\otimes}(\Disk^{1\fr},\Sp) \simeq \Alg_{\mathbb{E}_1}(\Sp)
\]
Moreover, the $\Ztwo$-action makes $A_H (\mathbb{R}^1)$ into an associative ($\mathbb{E}_1$) algebra with involution i.e. $A_H (\mathbb{R}^1)$ is a $A_H (\mathbb{R}^1)\otimes A_H (\mathbb{R}^1)$-module. To add on, we have seen that this $A_H (\mathbb{R}^1)\otimes A_H (\mathbb{R}^1)$-module structure lifts to $N^{\Ztwo}_{e} (A_H (\mathbb{R}^{1}))$-module structure on $A(\mathbb{R}^{\sigma})$, hence we say that $A_H$ is an associative algebra with genuine involution.

The universal property of  equivariant disc algebras, Theorem \ref{universalthm}, gives us the equivalence
\[
\Fun^{\otimes}_{\Ztwo} (\underline{\Disk}^{\Ztwo,\mathbb{R}^{\sigma}\fr}, \underline{\Sp}^{\Ztwo}) \simeq \Fun^{\otimes}_{O(2)}(\underline{\Disk}^{O(2),S^1\fr},\underline{\Sp}^{O(2)})
\]
where we have used the fact that $O(2)/\Ztwo \cong S^1$.

Let $A\in \Fun^{\otimes}_{O(2)}(\underline{\Disk}^{O(2),S^1\fr},\underline{\Sp}^{O(2)})$ be an $O(2)$-disc algebra corresponding to $A_H$. We can regard $S^1$ as a genuine $O(2)$-object (i.e. a coCartesian section) in the $O(2)$-$\infty$-category of manifolds $\underline{\Mfld}^{O(2),S^1\fr}$ given by
\begin{align*}
\underline{S^1} : & \mathcal{O}_{O(2)}^{\op} \rightarrow \underline{\Mfld}^{O(2),S^1\fr} \\
 O & \mapsto O\times S^1
\end{align*}
Therefore, we can build the following map
\begin{center}
\begin{tikzcd} [row sep=2em, column sep=3em]
\{ \text{Associative algebras with genuine involution in } \underline{\Sp}^{\Ztwo}\} \arrow[d] \\
\{ \text{Genuine $O(2)$-objects in } \underline{\Sp}^{O(2)} \}
\end{tikzcd}
\end{center}
given by
\[
A_H \mapsto \int_{S^1} A
\]
The result of Horev, \cite[7.1.1 and 7.1.2]{AH19} tells us that the underlying $\Ztwo$-genuine object of $\int_{S^1} A$ is equivalent to the real topological Hochschild homology $\THR(A)$ (where we have written $A$ for $A_{[O(2)/C_2]}$) (see also \cite{DMPR17}). In other words, we have
\[
\left(\int_{S^1} A \right)_{[O(2)/\Ztwo]} \simeq \THR(A)
\]
Therefore, we have obtained the refinement of the $\Ztwo$-genuine structure on $\THR(A)$ to the $O(2)$-genuine structure.

\subsection*{Associative algebras and $S^1$-genuine objects}

For the second example, let us take $G=S^1$ and $H=\{e\}$ the trivial subgroup. The adjoint representation of $S^1$ is equal to $\mathbb{R}^1$ when forgetting the action, hence the $\infty$-category $\Fun^{\otimes}(\underline{\Disk}^{e,\mathbb{R}\fr}, \underline{\Sp}^{S^1}_{[S^1/e]})$ is equivalent to the $\infty$-category of $\mathbb{E}_1$-algebra objects in $ \underline{\Sp}^{S^1}_[S^1/e] \simeq \Sp$. Therefore, the theorem \ref{universalthm} gives us
\[
\Alg_{\mathbb{E}_1}(\Sp) \simeq \Fun^{\otimes}_{S^1}(\underline{\Disk}^{S^1,S^1\fr},\underline{\Sp}^{S^1})
\]
Let $A^{\operatorname{ass}}_e$ be an associative algebra object in the $\infty$-category of spectra $\Sp$ which corresponds to the $S^1$-framed $S^1$-algebra object $A$ in $\underline{\Sp}^{S^1}$. Then we can construct a map
\begin{center}
\begin{tikzcd} [row sep=2em, column sep=3em]
\{ \text{Associative algebras in } \Sp\} \arrow[d] \\
\{ \text{Genuine $S^1$-objects in } \underline{\Sp}^{S^1} \}
\end{tikzcd}
\end{center}
given by
\[
A_e^{\operatorname{ass}} \mapsto \int_{S^1} A
\]
Where the $S^1$-factorization homology is taken with respect to $S^1$ as an $S^1$-genuine object given by
\begin{align*}
\underline{S}^1 : & \mathcal{O}_{S^1}^{\op} \rightarrow \underline{\Mfld}^{S^1,S^1\fr} \\
 O & \mapsto O\times S^1
\end{align*}
By \cite[7.2.2]{AH19} we have
\[
\left( \int_{S^1} A \right)_{[S^1/C_n]}^{\Phi C_n} \simeq \THH(A;A^\tau)
\]
where $A^{\tau}$ is a $A$-$A^{\op}$-bimodule given by the formula
\begin{align*}
& A\otimes A^\tau \otimes A  \rightarrow A^\tau \\
& x \otimes a \otimes y  \rightarrow \tau x \otimes a \otimes y
\end{align*}
with $\tau\in C_n$ being the group generator.
\addtocontents{toc}{\protect\setcounter{tocdepth}{2}}

\cleardoublepage

\end{document}